\documentclass[11pt,letterpaper,twoside,reqno]{amsart}

\usepackage{amsmath}
\usepackage{amssymb}
\usepackage{bm}
\usepackage{mathrsfs}
\usepackage[dvips]{graphicx}

\usepackage{color}

\usepackage{url}

\usepackage{supertabular}

\usepackage{alg}

\makeatletter
\def\algspacing{\alg@unmargin}
\makeatother

\newtheorem{thm}{Theorem}

\newtheorem{lemma}[thm]{Lemma}
\newtheorem{prop}[thm]{Proposition}

\theoremstyle{remark}

\DeclareMathAlphabet{\mathsfsl}{OT1}{cmss}{m}{sl}

\newcommand{\term}{\emph}

\newcommand{\cnst}[1]{\mathrm{#1}}

	{\makebox{\phantom{}} \\ %
		\textsc{Input:}\begin{itemize}}%
	{\end{itemize}}

	{\textsc{Output:}\begin{itemize}}%
	{\end{itemize}}

	{\textsc{Procedure:}\begin{enumerate}}%
	{\end{enumerate}}

\renewcommand{\phi}{\varphi}

\newcommand{\Id}{\mathbf{I}}

\newcommand{\coll}[1]{\mathscr{#1}}

\newcommand{\abs}[1]{\left\vert {#1} \right\vert}

\newcommand{\diff}[1]{\mathrm{d}{#1}}
\newcommand{\idiff}[1]{\, \diff{#1}}

\newcommand{\Expect}{\operatorname{\mathbb{E}}}

\newcommand{\vct}[1]{\bm{#1}}
\newcommand{\mtx}[1]{\bm{#1}}

\newcommand{\adj}{*}

\newcommand{\trace}{\operatorname{tr}}

\newcommand{\norm}[1]{\left\Vert {#1} \right\Vert}

\newcommand{\enorm}[1]{\norm{#1}_2}

\newcommand{\triplenorm}[1]{\left\vert\!\left\vert\!\left\vert {#1} \right\vert\!\right\vert\!\right\vert}

\newcommand{\bigO}{\mathrm{O}}

\usepackage{subfigure}
\usepackage{wasysym}

\begin{document}
\title[A Comparison Principle for Random Subspaces]
{A Comparison Principle for Functions \\
of a Uniformly Random Subspace}


\author{Joel~A.~Tropp}




\date{28 January 2011.  Revised 23 March 2011.}

\begin{abstract}
This note demonstrates that it is possible to bound the expectation of an arbitrary norm of a random matrix drawn from the Stiefel manifold in terms of the expected norm of a standard Gaussian matrix with the same dimensions.  A related comparison holds for any convex function of a random matrix drawn from the Stiefel manifold.  For certain norms, a reversed inequality is also valid.
\end{abstract}

\maketitle

\section{Main Result}

Many problems in high-dimensional geometry concern the properties of a random $k$-dimensional subspace of the Euclidean space $\mathbb{R}^n$.  For instance, the Johnson--Lindenstrauss Lemma~\cite{JL84:Extensions-Lipschitz} shows that, typically, the metric geometry of a collection of $N$ points is preserved when we project the points onto a random subspace with dimension $\bigO(\log N)$.  Another famous example is Dvoretsky's Theorem~\cite{Dvo61:Some-Results,Mil71:New-Proof,Bal97:Elementary-Introduction}, which states that, typically, the intersection between the unit ball of a Banach space with dimension $N$ and a random subspace with dimension $\bigO(\log N)$ is comparable with a Euclidean ball.

In geometric problems, it is often convenient to work with matrices rather than subspaces.  Therefore, we introduce the \term{Stiefel manifold},
$$
\mathbb{V}_{k}^n :=
	\{ \mtx{Q} \in \mathbb{M}^{n \times k} : \mtx{Q}^\adj \mtx{Q} = \Id \},
$$
which is the collection of real $n \times k$ matrices with orthonormal columns.  
The elements of the Stiefel manifold $\mathbb{V}_k^n$ are sometimes called \term{$k$-frames in $\mathbb{R}^n$}.  The range of a $k$-frame in $\mathbb{R}^n$ determines a $k$-dimensional subspace of $\mathbb{R}^n$, but the mapping from $k$-frames to subspaces is not injective.
  
 
It is easy to check that each Stiefel manifold is invariant under orthogonal transformations on the left and the right.
An important consequence is that the Stiefel manifold $\mathbb{V}_k^n$ admits an invariant Haar probability measure, which can be regarded as a uniform distribution on $k$-frames in $\mathbb{R}^n$.  A matrix $\mtx{Q}$ drawn from the Haar measure on $\mathbb{V}_{k}^n$ is called a \term{random $k$-frame in $\mathbb{R}^n$}.

It can be challenging to compute functions of a random $k$-frame $\mtx{Q}$.  The main reason is that the entries of the matrix $\mtx{Q}$ are correlated on account of the orthonormality constraint $\mtx{Q}^\adj \mtx{Q} = \Id$.  Nevertheless, if we zoom in on a small part of the matrix, the local correlations are very weak because orthogonality is a global property.  In other words, the entries of a small submatrix of $\mtx{Q}$ are effectively \emph{independent} for many practical purposes~\cite{Jia06:How-Many}.

As a consequence of this observation, we might hope to replace certain calculations on a random $k$-frame by calculations on a random matrix with \emph{independent} entries.  An obvious candidate is a matrix $\mtx{G} \in \mathbb{M}^{n \times k}$ whose entries are independent ${\rm N}(0, n^{-1})$ random variables.  We call the associated probability distribution on $\mathbb{M}^{n \times k}$ the \emph{normalized Gaussian distribution}.

Why is this distribution a good proxy for a random $k$-frame in $\mathbb{R}^n$?  First, a normalized Gaussian matrix $\mtx{G}$ verifies $\Expect (\mtx{G}^\adj \mtx{G}) = \Id$, so the columns of $\mtx{G}$ are orthonormal on average.  Second, the normalized Gaussian distribution is invariant under orthogonal transformations from the left and the right, so it shares many algebraic and geometric properties with a random $k$-frame.  Furthermore, we have a wide variety of methods for working with Gaussian matrices, in contrast with the more limited set of techniques available for dealing with random $k$-frames.

These intuitions are well established in the random matrix literature, and many authors have developed detailed quantitative refinements.  In particular, we mention Jiang's paper~\cite{Jia06:How-Many} and its references, which discuss the proportion of entries in a random orthogonal matrix that can be simultaneously approximated using independent standard normal variables.  Subsequent work by Chatterjee and E.~Meckes~\cite{CM08:Multivariate-Normal} demonstrates that the joint distribution of $k$ (linearly independent) linear functionals of a random orthogonal matrix is close in Wasserstein distance to an appropriate Gaussian distribution, provided that $k = {\rm o}(n)$.  


We argue that there is a general comparison principle for random $k$-frames and normalized Gaussian matrices of the same size.  Recall that a convex function is called \term{sublinear} when it is positive homogeneous.  Norms, in particular, are sublinear.  Theorem~\ref{thm:norm-comparison} ensures that the expectation of a nonnegative sublinear function of a random $k$-frame is dominated by that of a normalized Gaussian matrix.  This result also allows us to study moments and, therefore, tail behavior.

\begin{thm}[Sublinear Comparison Principle] \label{thm:norm-comparison}
Assume that $k = \rho n$ for $\rho \in (0,1]$.  Let $\mtx{Q}$ be uniformly distributed on the Stiefel manifold $\mathbb{V}_k^n$, and let $\mtx{G} \in \mathbb{M}^{n \times k}$ be a matrix with independent ${\rm N}(0, n^{-1})$ entries.  For each nonnegative, sublinear, convex function $\abs{\cdot}$ on $\mathbb{M}^{n \times k}$ and each weakly increasing, convex function $\Phi : \mathbb{R} \to \mathbb{R}$,
$$
\Expect \Phi( \abs{ \mtx{Q} } ) \leq \Expect \Phi( (1+\rho/2) \abs{ \mtx{G} } ).
$$
In particular, for all $k \leq n$,
$$
\Expect \Phi( \abs{ \mtx{Q} } ) \leq \Expect \Phi( 1.5 \abs{ \mtx{G} } ).
$$
\end{thm}

Note that the leading constant in the bound is asymptotic to one when $k = {\rm o}(n)$.  Conversely, Section~\ref{sec:examples} identifies situations where the leading constant must be at least one.  We establish Theorem~\ref{thm:norm-comparison} in Section~\ref{sec:proof} as a consequence of a more comprehensive result, Theorem~\ref{thm:full-comparison}, for convex functions of a random $k$-frame.  



A simple example suffices to show that Theorem~\ref{thm:norm-comparison} does not admit a matching lower bound, no matter what comparison factor $\beta$ we allow.  Indeed, suppose that we fix a positive number $\beta$.  Write $\norm{\cdot}$ for the spectral norm (i.e., the operator norm between two Hilbert spaces), and consider the weakly increasing, convex function
$$
\Phi(t) := \left((t)_+ - 1 \right)_+
\quad\text{where}\quad
(a)_+ := \max\{ 0, a \}.
$$
For a normalized Gaussian matrix $\mtx{G} \in \mathbb{M}^{n \times k}$, we compute that
$$
\Expect \Phi( \beta \norm{ \mtx{G} } )
	= \Expect \left( \beta \norm{ \mtx{G} } - 1 \right)_+
	> 0
$$
because there is always a positive probability that $\beta \norm{\mtx{G}} \geq 2$.
Meanwhile, the spectral norm of a random $k$-frame $\mtx{Q}$ in $\mathbb{R}^n$ satisfies $\norm{ \mtx{Q} } = 1$, so
$$
\Expect \Phi( \norm{ \mtx{Q} } )
	= \Expect \Phi( 1 )
	= 0.
$$
Inexorably,
$$
\Expect \Phi( \beta \norm{ \mtx{G} } )
	\leq \Expect \Phi( \norm{ \mtx{Q} } )
	\quad\Longrightarrow\quad
	\beta \leq 0.
$$
Therefore, it is impossible to control $\Phi(\beta \abs{\mtx{G}})$ using $\Phi(\abs{\mtx{Q}})$ unless we impose additional restrictions.  Turn to Section~\ref{sec:converse} for some conditions under which we can reverse the comparison in Theorem~\ref{thm:norm-comparison}.

One of the anonymous referees has made a valuable point that deserves amplification.  Note that a random orthogonal matrix with dimension one is a scalar Rademacher variable, while a normalized Gaussian matrix with dimension one is a scalar Gaussian variable.  From this perspective, Theorem~\ref{thm:norm-comparison} resembles a noncommutative version of the classical comparison between Rademacher series and Gaussian series in a Banach space~\cite[Sec.~4.2]{LT91:Probability-Banach}.  Let us state an extension of Theorem~\ref{thm:norm-comparison} that makes this connection explicit.


\begin{thm}[Noncommutative Gaussian Comparison Principle] \label{thm:nc-gauss-comparison}
Fix a sequence of square matrices $\{ \mtx{A}_j : j = 1, \dots, J \} \subset \mathbb{M}^{n \times n}$.  Consider an independent family $\{ \mtx{Q}_j : j = 1, \dots, J \} \subset \mathbb{M}^{n \times n}$ of random orthogonal matrices, and an independent family $\{ \mtx{G}_j : j = 1, \dots, J \} \subset \mathbb{M}^{n \times n}$ of normalized Gaussian matrices.  For each nonnegative, sublinear, convex function $\abs{\cdot}$ on $\mathbb{M}^{n \times n}$ and each weakly increasing, convex function $\Phi : \mathbb{R} \to \mathbb{R}$,
$$
\Expect \Phi\left( \abs{ \sum\nolimits_{j=1}^J \mtx{Q}_j \mtx{A}_j } \right)
	\leq \Expect \Phi\left( 1.5 \abs{ \sum\nolimits_{j=1}^J \mtx{G}_j \mtx{A}_j } \right).
$$
\end{thm}

We can complete the proof of Theorem~\ref{thm:nc-gauss-comparison} using an obvious variation on the arguments behind Theorem~\ref{thm:norm-comparison}. 
We omit further details out of consideration for the reader's patience.

\section{A Few Examples} \label{sec:examples}

Before proceeding with the proof of Theorem~\ref{thm:norm-comparison}, we present some applications that may be interesting.  We need the following result~\cite[Thm.~3.20]{LT91:Probability-Banach}, which is due to Gordon~\cite{Gor88:Gaussian-Processes}.

\begin{prop}[Spectral Norm of a Gaussian Matrix] \label{prop:gordon}
Let $\mtx{G} \in \mathbb{M}^{n \times k}$ be a random matrix with independent ${\rm N}(0, n^{-1})$ entries.  Then
$$
\Expect \norm{ \mtx{G} }_{\ell_2^k \to \ell_2^n}
	\leq 1 + \sqrt{k/n}. 
$$
\end{prop}




\subsection{How good are the constants?}
Consider a uniformly random orthogonal matrix $\mtx{Q} \in \mathbb{V}_n^n$.  Evidently, its spectral norm $\norm{ \mtx{Q} } = 1$.  Let $\mtx{G} \in \mathbb{M}^{n\times n}$ be a normalized Gaussian matrix.  Theorem~\ref{thm:norm-comparison} and Proposition~\ref{prop:gordon} ensure that
$$
1 = \Expect \norm{ \mtx{Q} } \leq 1.5 \Expect \norm{\mtx{G}} \leq 3.
$$
Thus, the constant 1.5 in Theorem~\ref{thm:norm-comparison} cannot generally be improved by a factor greater than three.

Next, we specialize to the trivial case where $k = n = 1$.  Let $Q$ be a Rademacher random variable, and let $G$ be a standard Gaussian random variable.  Theorem~\ref{thm:norm-comparison} implies that
$$
1 = \Expect \abs{ Q } \leq 1.5 \Expect \abs{ G }
	= 1.5 \sqrt{\frac{2}{\pi}}
	< 1.2.
$$
Therefore, we cannot improve the constant by a factor of more than 1.2 if we demand a result that holds when $n$ is small.

Finally, consider the case where $k = 1$.  Let $\vct{q}$ be a random unit vector in $\mathbb{R}^n$, and let $\vct{g}$ be a vector in $\mathbb{R}^n$ with independent ${\rm N}(0, n^{-1})$ entries.  Applying Theorem~\ref{thm:norm-comparison} with the Euclidean norm, we obtain
$$
1 = \Expect \enorm{ \vct{q} } \leq \left(1 + \frac{1}{2n}\right) \cdot \Expect \enorm{ \vct{g} }
	\leq 1 + \frac{1}{2n}.
$$
This example demonstrates that the best constant in Theorem~\ref{thm:norm-comparison} is at least one when $k = 1$ and $n$ is large.
Related examples show that the best constant is at least one as long as $k = {\rm o}(n)$.

\subsection{Maximum entry of a random orthogonal matrix}
Consider a uniformly random orthogonal matrix $\mtx{Q} \in \mathbb{V}_n^n$, and let $\mtx{G} \in \mathbb{M}^{n\times n}$ be a normalized Gaussian matrix.  Using Theorem~\ref{thm:norm-comparison} and a standard bound for the maximum of standard Gaussian variables, we estimate that 
$$
\Expect \max_{i,j} \abs{ \mtx{Q}_{ij} }
	\leq 1.5 \Expect \max_{i, j} \abs{ \mtx{G}_{ij} }
	\leq 1.5 \sqrt{\frac{ 2 \log(n^2) + 1 }{ n }}
	= 3 \sqrt{\frac{ \log(n) + 1/4}{ n }}
$$
Jiang~\cite{Jia05:Maxima-Entries} has shown that, almost surely, a sequence $\{\mtx{Q}^{(n)} \}$ of random orthogonal matrices with $\mtx{Q}^{(n)} \in \mathbb{V}_n^n$ has the limitng behavior
$$
\liminf_{n\to \infty}
\sqrt{\frac{n}{\log n}} \cdot \max_{i, j} \big\vert \mtx{Q}^{(n)}_{ij} \big\vert
= 2
\quad\text{and}\quad
\limsup_{n\to \infty}
\sqrt{\frac{n}{\log n}} \cdot \max_{i, j} \big\vert \mtx{Q}^{(n)}_{ij} \big\vert
= \sqrt{6}.
$$
We see that our simple estimate is not sharp, but it is very reasonable.

\subsection{Spectral norm of a submatrix of a random $k$-frame}
Consider a uniformly random $k$-frame $\mtx{Q} \in \mathbb{V}_{k}^n$, and let $\mtx{G} \in \mathbb{M}^{n \times k}$ be a normalized Gaussian matrix.  Define the linear map $\coll{L}_{j}$ that restricts an $n \times k$ matrix to its first $j$ rows and rescales it by $\sqrt{n/j}$.  As a consequence, the columns of the $j \times k$ matrix $\coll{L}_j(\mtx{Q})$ approximately have unit Euclidean norm.  We may compute that
$$
\Expect \norm{ \coll{L}_{j}(\mtx{Q}) }
	\leq (1+(k/2n)) \Expect \norm{ \coll{L}_{j}( \mtx{G} ) }
	\leq (1+ (k/2n)) (1 + \sqrt{k/j})
$$
because of Theorem~\ref{thm:norm-comparison} and Proposition~\ref{prop:gordon}.

This estimate is interesting because it applies for all values of $j$ and $k$.  Note that the leading constant $1 + (k/2n)$ is asymptotic to one whenever $k = {\rm o}(n)$.  In contrast, we recall Jiang's result~\cite{Jia06:How-Many} that the total-variation distance between the distributions of $\coll{L}_{j}(\mtx{Q})$ and $\coll{L}_{j}(\mtx{G})$ vanishes if and only if $j, k = {\rm o}(\sqrt{n})$.  A related fact is that, under a natural coupling of $\mtx{Q}$ and $\mtx{G}$, the matrix $\ell_\infty$-norm distance between $\coll{L}_{j}(\mtx{Q})$ and $\coll{L}_{j}(\mtx{G})$ vanishes in probability if and only if $k = {\rm o}(n / \log n)$.


%



\section{Proof of the Sublinear Comparison Principle} \label{sec:proof}

The main tool in our proof is a well-known theorem of Bartlett that describes the statistical properties of the {\sf QR} decomposition of a standard Gaussian matrix, i.e., a matrix with independent ${\rm N}(0, 1)$ entries.  See Muirhead's book~\cite{Mui82:Aspects-Multivariate} for a detailed derivation of this result.

\begin{prop}[The Bartlett Decomposition] \label{prop:bartlett}
Assume that $k \leq n$, and let $\mtx{\Gamma} \in \mathbb{M}^{n \times k}$ be a standard Gaussian matrix.  Then
$$
\mtx{\Gamma}_{n \times k} \sim \mtx{Q}_{n \times k} \, \mtx{R}_{k \times k}.
$$
The factors $\mtx{Q}$ and $\mtx{R}$ are statistically independent.
The matrix $\mtx{Q}$ is uniformly distributed on the Stiefel manifold $\mathbb{V}_{k}^n$.  The matrix $\mtx{R}$ is a random upper-triangular matrix of the form
$$
\mtx{R} = \begin{bmatrix}
X_1 & Y_{12} & Y_{13} & \dots & Y_{1k} \\
& X_2 & Y_{23} & \dots & Y_{2k} \\
&& \ddots & \ddots & \vdots \\
&&& X_{k-1} & Y_{k-1,k} \\
&&&& X_k
\end{bmatrix}_{k \times k}.
$$
where the diagonal entries $X_i^2 \sim \chi^2_{n-i+1}$ and the super-diagonal entries $Y_{ij} \sim {\rm N}(0, 1)$; furthermore, all these random variables are mutually independent. 
\end{prop}


We may now establish a comparison principle for a general convex function of a random $k$-frame.

\begin{thm}[Convex Comparison Principle] \label{thm:full-comparison}
Assume that $k \leq n$.  Let $\mtx{Q} \in \mathbb{M}^{n \times k}$ be uniformly distributed on the Stiefel manifold $\mathbb{V}_k^n$, and let $\mtx{\Gamma} \in \mathbb{M}^{n \times k}$ be a standard Gaussian matrix.  For each convex function $f : \mathbb{M}^{n \times k} \to \mathbb{R}$, it holds that
$$
\Expect f( \mtx{Q} ) \leq \Expect f( \alpha^{-1} \mtx{\Gamma} )
\quad\text{where}\quad
\alpha := \alpha(k, n) := \frac{1}{k} \sum\nolimits_{i=1}^k \Expect( X_{i} )
$$
and $X_{i}^2 \sim \chi^2_{n-i+1}$.  Similarly, for each concave function $g : \mathbb{M}^{n \times k} \to \mathbb{R}$, it holds that
$$
\Expect g( \mtx{Q} ) \geq \Expect g( \alpha^{-1} \mtx{\Gamma} ).
$$
\end{thm}

\begin{proof}
The result is a direct consequence of the Bartlett decomposition and Jensen's inequality.  Define $\mtx{\Gamma}$, $\mtx{Q}$, and $\mtx{R}$ as in the statement of Proposition~\ref{prop:bartlett}.  Let $\mtx{P} \in \mathbb{M}^{k \times k}$ be a uniformly random permutation matrix, independent from everything else.

First, observe that
$$
\Expect ( \mtx{P} \mtx{R} \mtx{P}^T )
	= (\Expect \bar{\trace}(\mtx{R})) \cdot \Id
	= \alpha \Id
\quad\text{where}\quad
\alpha := \frac{1}{k} \sum\nolimits_{i=1}^k \Expect( X_{i} ).
$$
The symbol $\bar{\trace}$ denotes the normalized trace, and the random variable $X_{i} \sim \chi^2_{n-i+1}$ for each index $i = 1, \dots, k$.
Since the function $f$ is convex, Jensen's inequality allows that
$$
\Expect f( \mtx{Q} )
	= \Expect f(  \alpha^{-1} \mtx{Q} (\Expect \mtx{PRP}^T ) )
	\leq \Expect f( \alpha^{-1} \mtx{QPRP}^T ).
$$
It remains to simplify the random matrix in the last expression.

Recall that the Haar distribution on the Stiefel manifold $\mathbb{V}^n_k$ and the normalized Gaussian distribution on $\mathbb{M}^{n \times k}$ are both invariant under orthogonal transformations.  Therefore, $\mtx{Q} \sim \mtx{QS}$ and $\mtx{\Gamma} \sim \mtx{\Gamma S}^T$ for each fixed permutation matrix $\mtx{S}$.  It follows that
$$
\Expect[ f( \alpha^{-1} \mtx{QPRP}^T ) \, | \, \mtx{P} ]
	= \Expect[ f( \alpha^{-1} \mtx{QRP}^T ) \, | \, \mtx{P} ]
	= \Expect[ f( \alpha^{-1} \mtx{\Gamma P}^T ) \, | \, \mtx{P} ]
	= \Expect f( \alpha^{-1} \mtx{\Gamma} ),
$$
where we have also used the fact that $\mtx{Q}$ and $\mtx{R}$ are statistically independent.
Combining the last two displayed formulas with the tower property of conditional expectation, we reach
$$
\Expect f( \mtx{Q} )
	\leq \Expect \Expect[ f( \alpha^{-1} \mtx{Q PRP}^T ) \, | \, \mtx{P} ]
	= \Expect f( \alpha^{-1} \mtx{\Gamma} ).
$$
The proof for concave functions is analogous.
\end{proof}

For Theorem~\ref{thm:full-comparison} to be useful, we need to make some estimates for the constant $\alpha(k, n)$ that arises in the argument.  To that end, we state without proof a simple result on the moments of a chi-square random variable.

\begin{prop}[Chi-Square Moments] \label{prop:chisq-mom}
Let $\Xi^2$ be a chi-square random variable with $p$ degrees of freedom.  Then
$$
\Expect( \Xi ) = \frac{\sqrt{2} \cdot \Gamma((p+1)/2)}{\Gamma(p/2)}.
$$
\end{prop}

Given the identity from Proposition~\ref{prop:chisq-mom}, standard inequalities for this ratio of gamma functions allow us to estimate the constant $\alpha$ in terms of elementary operations and radicals. 

\begin{lemma}[Estimates for the Constant] \label{lem:estimate-i}
The constant $\alpha(k, n)$ defined in Theorem~\ref{thm:full-comparison} satisfies
$$
\frac{1}{k} \sum\nolimits_{i=0}^{k-1}
	\sqrt{n - (i + 1/2)}
	\quad \leq \quad \alpha(k, n)
	\quad \leq \quad \frac{1}{k} \sum\nolimits_{i=0}^{k-1} \sqrt{n - i}.
$$
\end{lemma}

\begin{proof}
We require bounds for
$$
\alpha = \frac{1}{k} \sum\nolimits_{i=1}^{k} \Expect( X_i )
\quad\text{where}\quad
X_i^2 \sim \chi^2_{n-i+1}.
$$
Proposition~\ref{prop:chisq-mom} states that
$$
\Expect( X_i ) = \frac{\sqrt{2} \cdot \Gamma((p_i+1)/2)}{\Gamma(p_i/2)}
\quad\text{for $p_i = n - i + 1$.}
$$
This ratio of gamma functions appears frequently, and the following bounds are available.
$$
\sqrt{p- 1/2}
	< \frac{\sqrt{2} \cdot \Gamma((p + 1)/2)}{\Gamma(p/2)} 
	< \sqrt{p}
\quad\text{for $p \geq 1/2$.}
$$
Combine these relations and reindex the sums to reach the result.

The upper bound can be obtained directly from Jensen's inequality and the basic properties of a chi-square variable: $\Expect (X_i) \leq [\Expect (X_i^2) ]^{1/2} = \sqrt{n - i + 1}$.  In contrast, the lower bound seems to require hard analysis.
\end{proof}

For practical purposes, it is valuable to simplify the estimates from Lemma~\ref{lem:estimate-i} even more.  To accomplish this task, we interpret the sums in terms of basic integral approximations.

\begin{lemma}[Simplified Estimates] \label{lem:simple-estimate}
The constant $\alpha(k, n)$ defined in Theorem~\ref{thm:full-comparison} satisfies
$$
\frac{2}{3k} \left[ n^{3/2} - (n-k)^{3/2} \right]
	\quad \leq \quad \alpha(k, n)
	\quad \leq \quad \frac{2}{3k} \left[ n^{3/2} - (n - k)^{3/2} \right]
	+ \frac{1}{2k} \left[ \sqrt{n} - \sqrt{n - k} \right].
$$
The minimum value for the lower bound occurs when $k = n$, and
$$
\frac{2}{3} \sqrt{n}
	\quad\leq\quad \alpha(n, n)
	\quad\leq\quad \frac{2}{3} \sqrt{n} + {\rm o}(1)
	\quad\text{as $n \to \infty$.}
$$
Furthermore, when we express $k = \rho n$ for $\rho \in (0,1]$, it holds that
$$
\frac{1}{\alpha( \rho n, n )}
	\quad\leq\quad
	\frac{1}{\sqrt{n}} \cdot (1 + \rho / 2).
$$
\end{lemma}

\begin{proof}
Fix the parameters $k$ and $n$. Define the real-valued function $h(x) = \sqrt{n - x}$, and observe that $h$ is concave and decreasing on its natural domain.  The lower bound for $\alpha$ from Lemma~\ref{lem:estimate-i} implies that
$$
\alpha \geq \frac{1}{k} \sum\nolimits_{i=0}^{k-1} h(i + 1/2)
	\geq \frac{1}{k} \int_{0}^{k} h(x) \idiff{x}.
$$
To justify the second inequality, we observe that the sum corresponds with the midpoint-rule approximation to the integral.  Because the integrand is concave, the midpoint rule must overestimate the integral.  Evaluate the integral to obtain the stated lower bound.

To see that the minimum value for the lower bound occurs when $k = n$, notice that
$$
k \longmapsto \frac{1}{k} \int_0^k h(x) \idiff{x}
$$
is the running average of a decreasing function.  Of course, the running average also decreases.

Next, we use the relation $k = \rho n$ to simplify (the reciprocal of) the lower bound, which yields
$$
\frac{1}{\alpha(\rho n, n)}
	\leq 1.5 \cdot n^{-1/2} \cdot \frac{\rho}{1 - (1-\rho)^{3/2}}
	\leq n^{-1/2} \cdot (1 + \rho/2).
$$
The second inequality holds because the fraction is a convex function of $\rho$ on the interval $(0,1]$, so we may bound it above by the chord $\rho \mapsto (2+\rho)/3$ connecting the endpoints. 

The proof of the upper bound follows from a related principle: The trapezoidal rule underestimates the integral of a concave function.  Lemma~\ref{lem:estimate-i} ensures that
$$
\alpha \leq \frac{1}{k} \sum\nolimits_{i=0}^{k-1} h(i)
	\leq \frac{1}{k} \left[ \int_0^{k} h(x) \idiff{x}
	+ \frac{1}{2} \left( h(0) - h(k) \right) \right].
$$
Here, we have applied the trapezoidal rule on the interval $[0, k]$ and then redistributed the terms associated with the endpoints.  Evaluate the integral to complete the bound.
\end{proof}

We are now prepared to establish the main result.

\begin{proof}[Proof of Theorem~\ref{thm:norm-comparison}]
Let $\mtx{Q}$ be a random matrix distributed uniformly on the Stiefel manifold $\mathbb{V}_{k}^n$, and let $\mtx{G} \in \mathbb{M}^{n \times k}$ be a normalized Gaussian matrix.  We can write $\mtx{G} = n^{-1/2} \mtx{\Gamma}$ where $\mtx{\Gamma}$ is a standard normal matrix.

Suppose that $\abs{\cdot}$ is a nonnegative, sublinear, convex function and that $\Phi$ is a weakly increasing, convex function.  Then the function
$\mtx{M} \mapsto \Phi( \abs{ \mtx{M} } )$ is also convex.  Theorem~\ref{thm:full-comparison} demonstrates that
$$
\Expect \Phi( \abs{ \mtx{Q} } )
	\leq \Expect \Phi \big( \abs{ \alpha^{-1} \mtx{\Gamma} } \big)
	= \Expect \Phi \big( \alpha^{-1} \sqrt{n} \cdot \abs{\mtx{G}}  \big).
$$
For $k = \rho n$, Lemma~\ref{lem:simple-estimate} ensures that the constant $\alpha$ satisfies 
$$
\alpha^{-1} \sqrt{n} \leq 1 + \rho /2.
$$
Given that the function $\Phi$ is increasing and $\abs{\mtx{G}} \geq 0$, we conclude that
$$
\Expect \Phi( \abs{ \mtx{Q} } )
	\leq \Phi \big( \alpha^{-1} \sqrt{n} \cdot \abs{ \mtx{G} } \big)
	\leq \Phi( (1 + \rho/2) \cdot \abs{ \mtx{G} } ).
$$
This argument establishes the main part of the theorem.  To establish the remaining assertion, we simply assign $\rho = 1$, the maximum value allowed.
\end{proof}

\section{Partial Converses} \label{sec:converse}

There are at least a few situations where it is possible to reverse the inequality of Theorem~\ref{thm:norm-comparison}.  To develop these results, we record another basic observation about Gaussian matrices~\cite{Mui82:Aspects-Multivariate}.

\begin{prop}[Polar Factorization] \label{prop:polar}
Assume that $k \leq n$.  Let $\mtx{\Gamma} \in \mathbb{M}^{n \times k}$ be a standard Gaussian matrix.  Then
$$
\mtx{\Gamma}_{n \times k} \sim \mtx{Q}_{n \times k} \mtx{W}_{k \times k}.
$$
The factors $\mtx{Q}$ and $\mtx{W}$ are statistically independent.  The matrix $\mtx{Q}$ is uniformly distributed on the Stiefel manifold $\mathbb{V}_k^n$, and the matrix $\mtx{W}$ is the positive square root of a $k \times k$ Wishart matrix with $n$ degrees of freedom.
\end{prop}

The first converse concerns a \term{right operator ideal norm}; that is, a norm $\triplenorm{ \cdot }$ that satisfies the relation
$
\triplenorm{ \mtx{AB} } \leq \triplenorm{ \mtx{A} } \cdot \norm{ \mtx{B} },
$
where $\norm{ \cdot }$ is the spectral norm.

\begin{thm}[Partial Converse I] \label{thm:pc1}
Assume that $k = \rho n$ for $\rho \in (0,1]$.  Let $\mtx{Q}$ be uniformly distributed on the Stiefel manifold $\mathbb{V}_k^n$, and let $\mtx{G} \in \mathbb{M}^{n \times k}$ be a normalized Gaussian matrix.  For each right operator ideal norm $\triplenorm{ \cdot }$, it holds that
$$
\Expect \triplenorm{ \mtx{G} }
	\leq (1 + \sqrt{\rho}) \cdot \Expect \triplenorm{ \mtx{Q} }.
$$ 
\end{thm}

\begin{proof}
The proof uses the polar factorization of the Gaussian matrix described in Proposition~\ref{prop:polar}.  For a standard Gaussian matrix $\mtx{\Gamma} \in \mathbb{M}^{n \times k}$,
$$
\Expect \triplenorm{ \mtx{G} }
	= n^{-1/2} \Expect \triplenorm{ \mtx{\Gamma} }
	= n^{-1/2} \Expect \triplenorm{ \mtx{QW} }
	\leq n^{-1/2} \Expect (\triplenorm{ \mtx{Q} } \cdot \norm{\mtx{W}} )
	= n^{-1/2} (\Expect \triplenorm{ \mtx{Q} }) \cdot( \Expect \norm{\mtx{W}} ).
$$
The last relation relies on the independence of the polar factors.  To continue, we note that the Wishart square root $\mtx{W}$ has the same distribution as $(\mtx{\Gamma}^\adj \mtx{\Gamma})^{1/2}$.  Therefore,
$$
n^{-1/2} \Expect \norm{ \mtx{W} }
	= n^{-1/2} \Expect \big( \norm{ \mtx{\Gamma}^\adj \mtx{\Gamma} }^{1/2} \big)
	= n^{-1/2} \Expect \norm{ \mtx{\Gamma} }
	= \Expect \norm{ \mtx{G} }
	\leq 1 + \sqrt{k/n},
$$
where the last bound follows from Gordon's result, Proposition~\ref{prop:gordon}.
\end{proof}

A version of Theorem~\ref{thm:pc1} also holds for higher moments:
$$
\Expect( \triplenorm{ \mtx{G} }^m )
	\leq \Expect( \norm{ \mtx{G} }^m ) \cdot \Expect( \triplenorm{ \mtx{Q} }^m )
	\leq \cnst{C} \sqrt{m} \cdot (1 + \sqrt{\rho})
	\cdot \Expect( \triplenorm{ \mtx{Q} }^m )
	\quad\text{when $m \geq 1$}.
$$
The second inequality holds because moments of a Gaussian series are equivalent~\cite[Cor.~3.2]{LT91:Probability-Banach}.

We have a second result that holds for other types of operator norms.
We omit the proof, which, by now, should be obvious.

\begin{thm}[Partial Converse II]
Assume that $k \leq n$.  Let $\mtx{Q}$ be uniformly distributed on the Stiefel manifold $\mathbb{V}_k^n$, and let $\mtx{G} \in \mathbb{M}^{k \times n}$ be a normalized Gaussian matrix.  Suppose that $\norm{\cdot}_Y$ is a norm on $\mathbb{R}^k$ and $\norm{ \cdot }_Z$ is a norm on $\mathbb{R}^n$.  Then
$$
\Expect \norm{ \mtx{G} }_{Y \to Z}
	\leq (n^{-1/2} \Expect \norm{ \mtx{T} }_{Y \to Y} )
	\cdot (\Expect \norm{ \mtx{Q} }_{Y \to Z} )
$$
where $\mtx{T}$ is either the upper-triangular matrix $\mtx{R}$ defined in Proposition~\ref{prop:bartlett} or the Wishart square root $\mtx{W}$ defined in Proposition~\ref{prop:polar}.
\end{thm}


\section*{Acknowledgments}

The author would like to thank Ben Recht and Michael Todd for encouraging him to refine and present these results.  Alex Gittens and Tiefeng Jiang provided useful comments on a preliminary draft of this article.  The anonymous referees offered several valuable comments.  This work has been supported in part by ONR awards N00014-08-1-0883 and N00014-11-1-0025, AFOSR award FA9550-09-1-0643, and a Sloan Fellowship. Some of the research took place at Banff International Research Station (BIRS).

\bibliographystyle{alpha}
\bibliography{unitary}

\end{document}